\documentclass[11pt,reqno,a4paper]{amsart}

\usepackage{amssymb,amsfonts,amsmath,amsthm}
\usepackage{mathrsfs}
\usepackage{graphicx}
\usepackage{color}
\usepackage{verbatim}
\usepackage{hyperref}
\usepackage{pdfsync}
\usepackage{todonotes}
\usepackage[all]{xy}

\newtheorem{thm}{Theorem}[section]
\newtheorem{pro}[thm]{Proposition}
\newtheorem{lem}[thm]{Lemma}
\newtheorem{cor}[thm]{Corollary}

\theoremstyle{definition}
\newtheorem{defn}[thm]{Definition}
\newtheorem{exa}[thm]{Example}

\newtheorem*{ackn}{Acknowledgements}

\theoremstyle{remark}
\newtheorem{rmk}[thm]{Remark}

\newtheorem{que}[thm]{Question}

\numberwithin{equation}{section}

\def\N{\mathbb{N}}
\def\J{\mathscr{J}}
\def\D{\mathscr{D}}
\def\R{\mathscr{R}}
\def\L{\mathscr{L}}
\def\H{\mathscr{H}}

\def\cK{\mathcal{K}}

\def\La{\Leftarrow}
\def\Ra{\Rightarrow}
\def\Lra{\Leftrightarrow}
\def\ol#1{\overline{#1}}

\def\gen#1{\langle{#1}\rangle}
\def\pre#1#2{\langle{#1}\,|\,{#2}\rangle}

\def\Sch{Sch\"u\-tzen\-ber\-ger }
\def\gem{group-em\-bed\-d\-a\-ble }

\DeclareMathOperator\Mon{Mon} \DeclareMathOperator\Inv{Inv} \DeclareMathOperator\Gp{Gp} 
\DeclareMathOperator\pref{pref}  \DeclareMathOperator\red{red}

  \DeclareMathOperator\Stab{Stab}
\DeclareMathOperator\MonRC{MonRC}

\begin{document}


\title[Prefix monoids and right units]%
{Prefix monoids of groups and right units of special inverse monoids} 





\maketitle

\begin{center}
IGOR DOLINKA
\footnote{
Department of Mathematics and Informatics, University of Novi Sad, Trg Dositeja Obra\-do\-vi\-\'ca 4,
21101 Novi Sad, Serbia.
\\ \emph{Email address:} \texttt{dockie@dmi.uns.ac.rs}.}
\ and \ 
ROBERT D. GRAY\footnote{School of Mathematics, University of East Anglia, Norwich NR4 7TJ, England.
\\ \emph{Email address:} \texttt{Robert.D.Gray@uea.ac.uk}.} 
\end{center}


\begin{abstract}
A prefix monoid is a finitely generated submonoid of a finitely presented group generated by the prefixes of its defining relators.  
Important results of Guba (1997), and of Ivanov, Margolis and Meakin (2001), show how the word problem for certain one-relator monoids, and inverse monoids, 
can be reduced to solving the membership problem in prefix monoids of certain one-relator groups. Motivated by this, 
in this paper we study the class of prefix monoids of finitely presented groups. We obtain a complete description of this class of monoids.  
All monoids in this family are finitely generated, recursively presented and group-embeddable.  Our results show that not every 
finitely generated recursively presented group-embeddable monoid is a prefix monoid, but for every such monoid if we take a free product 
with a suitably chosen free monoid of finite rank, then we do obtain a prefix monoid. Conversely we prove that every prefix monoid arises in this way. 
Also, we show that the groups that arise as groups of units of prefix monoids are precisely the finitely generated recursively presented groups,
while the groups that arise as \Sch groups of prefix monoids are exactly the recursively enumerable subgroups of finitely presented groups. 
We obtain an analogous result classifying the \Sch groups of monoids of right units of special inverse monoids.  
We also give some examples of right cancellative monoids arising as monoids of right units of finitely presented special inverse monoids, 
and show that not all right cancellative recursively presented monoids belong to this class.  
\end{abstract}




\section{Introduction}

The main themes of the present paper draw inspiration from the beautiful algebraic theory surrounding one of the longest standing and most important open problems  in combinatorial algebra and semigroup theory: the word problem for one-relator monoids. The classical result of Magnus \cite{Ma2} that all one-relator groups have algorithmically decidable word problems, as well as the method \cite{Ma1} it entails, is one of the  cornerstones of combinatorial group theory \cite{LSch}. On the other hand, the question of decidability the word problem for one-relator monoids has been a subject of a series of serious attacks, most notably by Adian \cite{Adj} and his students, among others. We refer to the survey \cite{CF} for a thorough historical account of this fascinating topic.

There have been notable partial successes for the word problem of one-relator monoids, perhaps the most glaring being the case of \emph{special} one-relator monoids, those of the form  $\Mon\pre{A}{w=1}$, where the word problem is always decidable. This result was proved by Adjan \cite{Adj} who also showed that these one-relator monoids have several other good algebraic properties. For example, he proved that the group of units of a one-relator special monoid $M$ is a one-relator group. Later, Makanin \cite{Mak} generalised this by showing that for any special monoid $M$, that is,  one given by defining relations of the form $w_i=1$, $i\in I$, the corresponding group of units can be defined by no more relations than in the original presentation  for $M$. Furthermore, all maximal subgroups of a special monoid $M$ must be isomorphic to its group of units \cite{Malh}, while the submonoids of \emph{right units} (right invertible elements) always turn out to be free products of the group of units and a free monoid (of finite rank, in the case when $M$ is finitely presented). See also the related work of Lallement \cite{Lal} and Zhang \cite{Zh1,Zh2}. 

Today we know, due to Adian and Oganessian \cite{AO}, that the general one-relator problem for monoids reduces to the cases of the form $\Mon\pre{a,b}{aub=ava}$  and $\Mon\pre{a,b}{aub=a}$, where $u,v\in \{a,b\}^*$ are arbitrary words. All monoids in these reduced cases are right cancellative, and this spurred  Ivanov, Margolis and Meakin \cite{IMM} to introduce yet another type of algebraic structure into the discourse: \emph{inverse monoids}, more precisely  \emph{special} inverse monoids, ones given by presentations of the form $\Inv\pre{A}{w_i=1\; (i\in I)}$. By embedding a right cancellative one-relator monoid into a one-relator special inverse monoid, they proved that the word problem for (ordinary) one-relator monoids reduces to the word problem for one-relator special inverse monoids. However, surprisingly, the second-named author of this paper has recently shown in \cite{Gr-Inv} that there \emph{exist} one-relator special inverse monoids in which the  word problem is undecidable. Still, this does not invalidate the approach suggested in \cite{IMM}, as the reduction presented there always produces a reduced---albeit not cyclically reduced---relator word, while the relator words constructed in \cite{Gr-Inv} are necessarily non-reduced.

The connection between prefix monoids of groups and the submonoid of right units (i.e. right invertible elements) of special inverse monoids  comes from the fact that the natural surjective homomorphism from a special inverse monoid $T=\Inv\pre{A}{w_i=1 \; (i \in I)}$ to its maximal group image $G=\Gp\pre{A}{w_i=1 \; (i \in I)}$ maps the submonoid $R$ of right units of $T$ surjectively to the \emph{prefix monoid} of $G$, with respect to this presentation. Here the prefix monoid $P$ is the submonoid of $G$ generated by all elements represented by the prefixes of the relator words $w_i \; (i \in I)$. In general this homomorphism from $R$ onto $P$ is not injective, but it is in the case that $T$ is an \emph{$E$-unitary} inverse monoid.  There are many equivalent definitions, but for our purposes it is the best to express it in the following way: $T=\Inv\pre{A}{w_i=1 \; (i \in I)}$ is $E$-unitary if in the natural homomorphism from $T$ to its greatest  group image $G=\Gp\pre{A}{w_i=1 \; (i \in I)}$ the only pre-images of the identity element of $G$ are the idempotents of $T$. Now, yet another major result of \cite{IMM}  shows that when $T=\Inv\pre{A}{w_i=1 \; (i \in I)}$ is $E$-unitary then, provided $\Gp\pre{A}{w_i=1 \; (i \in I)}$ has decidable word problem, the decidability of the word problem of $T$ reduces to the membership problem for the prefix monoid of $G$. In the same paper they also prove that when $w$ is a cyclically reduced word then $\Inv\pre{A}{w=1}$ is $E$-unitary and thus the word problem reduces to the prefix membership problem for the corresponding one-relator group in this case.   
The prefix membership problem for one-relator  groups have been studied by several authors, e.g.\ in \cite{Juh,MMSu}, and more recently by the present authors in \cite{DG}. The close connections between prefix monoids of groups and right units of special inverse monoids described above mean that it is natural to investigate both classes in parallel. 

Further motivation for the study of the right units of special inverse monoids comes from the fact that the main undecidability results \cite{Gr-Inv} are proved by constructing special inverse monoids in which membership in the submonoid of right units is undecidable. 
Also, the relevance of the study of prefix monoids is highlighted by important work of Guba \cite{Guba}. As mentioned above the word problem remains open for monoids of the form $\Mon\pre{a,b}{a=aub}$. 
It follows from the results of Guba \cite{Guba} that the word problem for one-relator monoids of the form $\Mon\pre{a,b}{a=aub}$, $u\in\{a,b\}^*$, reduces to the prefix membership problem in one-relator groups of the form $\Gp\pre{x,y,C}{xUyx^{-1}=1}$ where $x \neq y$, $C$ is a finite alphabet, and $U$ is a positive word over $\{x,y\} \cup C$. The prefix membership problem for this class of one-relator group presentations remains open.      

For brevity we shall use the expression \emph{prefix monoid} to always mean a monoid isomorphic to the prefix monoid of some finitely presented group, and by an \emph{RU-monoid} we mean one that is isomorphic to the submonoid of right units of some finitely presented special inverse monoid.   
In this paper we are shall consider the following four questions:
\begin{itemize}
\item[(1)] Which monoids are prefix monoids?
\item[(2)] Which monoids are RU-monoids?
\item[(3)] What can the groups of units of these monoids be?
\item[(4)] What can the \Sch groups of these monoids be?
\end{itemize}
For prefix monoids we shall obtain complete answers to all these questions, that is, we answer (1), (3) and (4) in this case. By the very definition, prefix monoids are all finitely generated, recursively presented and group-embeddable. Thus it is natural to wonder whether the answer to (1) could be exactly the class of finitely generated, recursively presented and group-embeddable monoids. It turns out this is not the case. Our results will show that not every finitely generated recursively presented group-embeddable monoid is a prefix monoid, but for every such monoid if we take a free product with a suitably chosen free monoid of finite rank, then we do obtain a prefix monoid. Conversely we prove that every prefix monoid arises in this way, thus answering question (1); see Theorem~\ref{thm:almost-all}. Concerning the subgroups of prefix monoids, since any \gem (and, more generally, any right cancellative) monoid  can have only a single idempotent, namely the identity element, it follows that the only maximal subgroup this the prefix monoid has is its group of units.  We are going to answer (3) for prefix monoids by showing (in Theorem~\ref{thm:units})  that the class of groups of units of prefix monoids is precisely the class of finitely generated recursively presented groups (which, by the Higman Embedding Theorem \cite{Hig}, is precisely  the class of all finitely generated subgroups of finitely presented groups). Lacking  any other subgroups, semigroup theory provides us with ``hidden'' group structures within arbitrary monoids, the so-called \emph{\Sch groups} (defined in a precise fashion  in the next section), which are generalisations of the notion of maximal subgroups to all (including non-regular) $\D$-classes of a semigroup. We shall answer question (4) for prefix monoid in Theorem~\ref{thm:Sch-prefix} where we show that the class of \Sch groups of all prefix monoids is precisely the class of recursively enumerable subgroups of finitely presented groups. 

The structure of RU-monoids is more complex than that of prefix monoids. RU-monoids were studied in \cite{GR} where it was shown that the RU-monoid of a one-relator special inverse monoid need not be finitely presented, and need not decompose as a free product of the groups of units and a free monoid (as they do in the case of non-inverse special monoids). By definition all RU-monoids are right cancellative and recursively presented. We shall see that not every right cancellative recursively presented monoid is an RU-monoid. Conversely we shall identify a large class of monoids that do all arise as RU monoids, called \emph{finitely RC-presented monoids}.  While the problem of completely describing RU-monoids remains open, we do succeed in classifying their \Sch groups.  For the groups of units of RU-monoids this question was answered in the recent paper \cite[Theorem 4.1]{GK} where the groups of units of finitely presented inverse monoids are shown to be precisely the finitely generated recursively presented groups.  Since the group of units of the right units $U$ of a special inverse monoid $M$ is equal to the group of units of the submonoid of right units $R$ of $M$, it follows from \cite[Theorem 4.1]{GK} that the groups of units of RU-monoids are exactly the finitely generated recursively presented groups. So that result answers question (3) for RU-monoids. In our final main result (Theorem~\ref{thm:Sch-ru}) we answer question (4) for RU-monoids showing that the class of \Sch groups of all prefix monoids is precisely the class of recursively enumerable subgroups of finitely presented groups. Hence while prefix monoids and RU-monoids are in general wildly different classes of monoids, on the level of groups of units and \Sch groups the possible behaviours coincide.  

The remainder of the paper consists of four sections beyond this introduction: one of preliminary nature where we accumulate the necessary notions and prerequisites, followed by two sections concerned with characterisations of prefix monoids and their associated groups, respectively, and the final section dealing with RU-monoids.


\section{Preliminaries}

\subsection{Presentations, prefix monoids, RU-monoids}

In the course of working with monoids, inverse monoids and groups, and presentations thereof, we shall be concerned with alphabets representing their generating sets 
(finite or infinite), such as $A$ in the monoid case and ``doubled alphabets'' $\ol{A}=A\cup A^{-1}$ in the inverse monoid and group case. To represent elements of 
these structures, we use words over these alphabets. The \emph{free monoid} $A^*$ consists of all words (finite sequences of letters) over $A$. The \emph{free group} 
$FG(A)$ on $A$ is defined on a subset of $\ol{A}^*$ whose elements are called \emph{reduced words}: these are words that do not contain subwords of the form $aa^{-1}$ 
and $a^{-1}a$ for any $a\in A$. Every word $w\in\ol{A}^*$ has its \emph{reduced form} $\red(w)$ obtained by successively removing subwords of the indicated form; 
the reduced form of a word is unique as this process can be shown to be confluent. So, in $FG(A)$, $u\cdot v=\red(uv)$. 
For any letter $a \in A$ we define $(a^{-1})^{-1}=a$ and then for any word $a_1 \ldots a_k \in \ol{A}^*$  we define $(a_1 \ldots a_k)^{-1} = a_k^{-1} \ldots a_1^{-1}$. With this notation if $w$ is a reduced word then $w^{-1}$ is the unique reduced word representing the inverse of $w$ in the free group.   
Finally, the \emph{free inverse monoid} $FIM(A)$ is obtained as the quotient of the free monoid $\ol{A}^*$ 
by the so-called \emph{Wagner congruence}, generated by the pairs $(uu^{-1}u,u)$ and $(uu^{-1}vv^{-1},vv^{-1}uu^{-1})$ for all 
$u,v\in \ol{A}^*$. An elegant geometric description of $FIM(A)$, with elements represented as finite connected birooted subgraphs of the Cayley graph of $FG(A)$, 
was given by Munn \cite{Munn} (see also Scheiblich \cite{Sch}), thus explaining the term \emph{Munn trees} used for such graphs.

Combinatorial algebra studies algebraic structures by representing then using \emph{presentations}: as quotients of corresponding free structures by certain congruences whose generating pairs are written as ``defining relations'' (or \emph{relators}). So, for a monoid we write $M=\Mon\pre{A}{u_i=v_i\; (i\in I)}$  if $M\cong A^*/\theta$ where $\theta$ is the congruence of the free monoid $A^*$ generated by the set of pairs of words $\{(u_i,v_i):\ i\in I\}$. Analogously $T=\Inv\pre{A}{u_i=v_i\; (i\in I)}$ is defined to be the quotient of the free inverse monoid $FIM(A)$ by the congruence on $FIM(A)$ generated by the set of pairs of words $\{(u_i,v_i):\ i\in I\}$. Finally, for  a group $G$ we write $G=\Gp\pre{A}{w_i=1\; (i\in I)}$ if $G\cong FG(A)/N$ where $N$ is a normal subgroup of $FG(A)$ generated (as a normal subgroup) by the words  $w_i$, $i\in I$ (in fact, to be precise, by the words $\red(w_i)$ as some of the $w_i$ might not be reduced). By a finitely presented monoid, inverse monoid, or group, we mean one that admits a presentation with finitely many generators and finitely many defining relations. It is easy to prove that a group is finitely presented as a group if and only if it is finitely presented as a monoid if and only if it is finitely presented as an inverse monoid. Interestingly, the same is not true for monoids and inverse monoids since even free inverse monoids are not finitely presented as monoids; see \cite{Schein1975}.

In general, for the element represented by a word $w$ (over the suitable alphabet) in the structure $S$ given by the presentation $\pre{A}{\mathfrak{R}}$ (be it  a monoid, an inverse monoid, or a group) we write $[w]_S$.

For further background in combinatorial group theory (such as free products and HNN extensions) we refer to \cite{LSch}.

\begin{defn}
Let $G$ be a finitely presented group. We say that $M$ is a \emph{prefix monoid in $G$} if there exists a finite presentation of $G$,
$$
G = \Gp\pre{A}{w_i=1\; (i\in I)},
$$
such that $M$ is isomorphic to the submonoid of $G$ generated by all elements $[p]_G$ such that the word $p$ is a prefix of $w_i$ for some $i\in I$. We simply say 
that $M$ is a \emph{prefix monoid} if it is a prefix monoid in some finitely presented group. Notice that a group might have a number of prefix monoids, as this 
definition depends on the presentation, not just the group it defines; it might happen 
(see Example~\ref{Ex:New})
that by changing a presentation of a given group we arrive at a different 
prefix monoid. 
For a fixed group presentation we talk about \emph{the} prefix monoid of the group with respect to that presentation.  
\end{defn}

\begin{rmk}
It is important to note that, since we are often concerned with considering group presentations and inverse monoid presentations given by the same generators and relators,  
here we do not assume that the relator words appearing in group presentations are cyclically reduced or even reduced.
Also, when considering inverse monoid presentations that turn out to define groups we do not necessarily assume that the relations $aa^{-1}=a^{-1}a=1$ for $a\in A$ are in the presentation; 
the fact that an element represented by a letter is invertible might be shown in an entirely different way e.g.\ by deducing relations of the form $au=va=1$ for some words $u,v\in \ol{A}^*$.
\end{rmk}

\begin{exa}\label{Ex:New}
The following example, taken from \cite[page 90]{IMM}, shows that the prefix monoid can depend on the choice of presentation for the group. 
The prefix monoid of the group presentation 
\[
G = \Gp\pre{a,b}{aba=1} 
\]
is generated by $\{a, ab\} = \{a, a^{-1}\}$ which, since $b=a^{-2}$ in this group, is the whole group $G$. On the other hand, the prefix monoid of the presentation  
\[
G = \Gp\pre{a,b}{aab=1} 
\]
is generated by $\{a, a^2 \}$ and hence is equal to the submonoid of $G$ generated by $\{ a\}$. 
Since $b=a^{-2}$ this generator is redundant and eliminating it we see that $G$ is isomorphic to the infinite cyclic group generated by $\{a\}$. So in the latter case the prefix monoid is isomorphic to the infinite monogenic monoid $\{a^i : i \geq 0 \}$. 
Hence these two different presentations for $G$ yield two different prefix monoids. 
\end{exa}

It is immediately clear that all prefix monoids are finitely generated and group-em\-bed\-d\-a\-ble. So, the following general question arises naturally.

\begin{que}
Which finitely generated \gem monoids arise as prefix monoids?
\end{que}

\begin{defn}
Let $M$ be a finitely presented special inverse monoid,
$$
M = \Inv\pre{A}{w_i=1\; (i\in I)}.
$$
We call the submonoid 
$
R = \{ m \in M: mm^{-1}=1 \} 
$
of right units of $M$ the \emph{RU-monoid} of $M$. A monoid $T$ is an \emph{RU-monoid} if it is isomorphic to the RU-monoid of some finitely presented special inverse monoid.
\end{defn}

It follows from the argument in the proof of \cite[Proposition 4.2]{IMM} that the RU-monoid of  $M = \Inv\pre{A}{w_i=1\; (i\in I)}$ is generated by the set of all elements $[p]_M$ represented by prefixes $p$ of the set of defining relators $w_i$, $i\in I$. Of course the RU-monoid of a special inverse monoid $M$ does not depend on the choice of presentation for $M$, but different choices of special presentation can give different finite generating sets for the monoid given by the prefixes of the defining relators. Since an RU-monoid is by definition an RU-monoid of a finitely presented special inverse monoid, it follows that the set of prefixes of defining relators is finite, hence all RU-monoids are finitely generated, and it is very easy to see that they are necessarily right cancellative. Indeed, if $a,b,c$ are right units of any monoid then, choosing $x$ to be an element satisfying $cx=1$ we see that if  $ac=bc$ then $a=acx=bcx=b$. So, we instantly have an analogous question to the one before.

\begin{que}
Which finitely generated right cancellative monoids arise as RU-monoids?
\end{que}

\subsection{Green's relations and \Sch groups}

The most basic tool in studying the structure of semigroups are the five equivalence relations called \emph{Green's relations}. The first three of them,
$\R,\L,\J$, classify the elements of a semigroup according to the right/left/two-sided principal ideals they generate. Since we only work with monoids in this paper we give the definitions here just for monoids. So, in a monoid $S$ we have the  following definitions:
$$
a\,\R\,b \Lra aS=bS, \quad a\,\L\,b \Lra Sa=Sb, \quad a\,\J\,b \Lra SaS=SbS.
$$
Further, $\H=\R\cap\L$ and $\D=\R\vee\L$, which is just $\R\circ\L$ as it may be shown that $\R$ and $\L$ commute. We say that $a\in S$ is \emph{regular} if $a=axa$ for some $x\in S$. The $\D$-classes are exclusive with respect to regularity: either all elements of a $\D$-class are regular, or none of them. In the former case, 
each $\R$-class and each $\L$-class contains at least one idempotent. 
The $\H$-classes containing idempotents are maximal subgroups of a monoid and, in turn, every maximal subgroup arises in this way. Group $\H$-classes within 
a given $\D$-class are all isomorphic; hence, there is a natural way to associate a group with each regular $\D$-class. Also, note that by the definition of $\R$ it is immediate that in any monoid $M$, an element $m$ belongs to the $\R$-class of $1$ in $M$ if and only if $m$ is right invertible. In particular the RU-monoid of a special inverse monoid is equal to the $\R$-class of $1$ in that monoid.  

Even though non-regular $\D$-classes do not contain idempotents and hence do not contain any subgroups, they carry a ``hidden'' group structure encapsulated by the concept of a \emph{\Sch group}. Namely, let $H$ be an $\H$-class within a $\D$-class $D$, say of a monoid $S$. Define first the \emph{right stabiliser} of $H$,
$$
\Stab(H) = \{s\in S:\ Hs\subseteq H\}.
$$
Then, if $s\in \Stab(H)$, it turns out that the right translation corresponding to the element $s$, $\rho_s:a\mapsto as$ ($a\in S$), restricts to a permutation of the set $H$. It is not difficult to see that $\Stab(H)$ is a submonoid of $S$ and upon defining an equivalence relation $\sigma$ on $\Stab(H)$ by $(s,t)\in\sigma$ if and only if $\rho_s|_H=\rho_t|_H$ then $\sigma$ is a congruence and the quotient $\Stab(H)/\sigma$ is a group. 
This is the \emph{right \Sch group} $\Gamma(H)$ of $H$. Analogously we can define the the \emph{left} \Sch group of an $\H$-class, but it may be shown that it is always isomorphic to  the right one. This group has $|H|$ elements, and \Sch groups of $\H$-classes from the same $\D$-class are isomorphic to each other. As is expected,  the \Sch group of an $\H$-class $H$ from a regular $\D$-class is just isomorphic to the maximal subgroup contained in that class; in particular, if an $\H$-class $H$ is a group then $\Gamma(H)\cong H$. 

For further background in semigroup theory we refer to \cite{CP,How}, and, specifically for inverse semigroups, to \cite{Law,Pet}.

\subsection{Recursive enumerability and the Higman Embedding Theorem}

A set of natural numbers $A\subseteq\N$ is \emph {recursively enumerable} (\emph{r.e.}\ for short) if it is an image of a unary (primitive) recursive function,
that is, $A=\{\varphi(n):\ n\in\N\}$ for some primitive recursive $\varphi:\N\to\N$. The notion of recursive enumerability can be extended to languages as well 
(and then also to sets of tuples of words) via some of the standard enumerations of all words over a finite alphabet. 
It is one of the principal results of Turing \cite{Tur,Mal} that r.e.\ languages are precisely the languages of Turing machines: $L\subseteq\Sigma^*$ is r.e.\ if 
and only there exists a Turing machine $\mathcal{M}$ such that $L(\mathcal{M})=L$, which means that $\mathcal{M}$ halts and accepts on any input word $w\in L$, 
while if $w\not\in L$ then $\mathcal{M}$ either halts and rejects the word or works forever. Actually, this is the criterion of a language being r.e.\ that we 
use throughout as is the case in modern theoretical computer science \cite{HMU}.

Now we say that a finitely generated group $G$ is \emph{recursively presented} if 
$$
G = \Gp\pre{A}{w_i=1\; (i\in I)}
$$
where $A$ is a finite set and $\{w_i:\ i\in I\}$ is a r.e.\ language over $\ol{A}=A\cup A^{-1}$. Similarly, a monoid $M$ is \emph{recursively presented} if
$$
M=\Mon\pre{A}{u_i=v_i\; (i\in I)}
$$
for a finite set $A$ and a r.e.\ subset $\{(u_i,v_i):\ i\in I\}$ of $A^*\times A^*$.

\begin{thm}[The Higman Embedding Theorem \cite{Hig}]
A finitely generated group embeds into a finitely presented group if and only if it is recursively presented.
\end{thm}

In fact, there is a counterpart of the previous theorem for semigroups/monoids and inverse semigroups/monoids, obtained, respectively, by Murski\u{\i} \cite{Mur} 
and Belyaev \cite{Bel}, so that a finitely generated (inverse) monoid embeds into a finitely presented one if and only if it is recursively presented.

Let $G = \Gp\pre{A}{w_i=1\; (i\in I)}$ is a finitely presented group. and that $L\subseteq\ol{A}^*$ is a r.e.\ language such that $H=\{[w]_G:\ w\in L\}$ is a subgroup of $G$. Then we say that $H$ is a \emph{recursively enumerable subgroup} of $G$. It is not difficult to see that any finitely generated subgroup of $G$ is recursively enumerable: if the generating elements of such a subgroup are represented by words $u_1,\dots,u_k\in\ol{A}^*$ then the (rational) language  $L=\{u_1,\dots,u_k\}^*$ suffices to see this. Therefore, by Higman's Theorem any finitely generated recursively presented group arises as a recursively enumerable subgroup of some finitely presented group. In fact, although we will not need it here, more generally it follows from the argument in \cite[Corollary to Theorem~1]{Hig} that any not necessarily finitely generated recursively presented group is isomorphic to a recursively enumerable subgroup of a finitely presented group.  

%
\section{A characterisation of prefix monoids}

By their very definition, prefix monoids are finitely generated submonoids of finitely presented groups. Therefore, our first aim is to take a closer look
at such monoids. It turns out that the following holds.

\begin{pro}\label{pro:mon-gr}
A finitely generated monoid embeds into a finitely presented group if and only if it is \gem and recursively presented.
\end{pro}

To prove this, we use an auxiliary result that is often labelled as folklore (such as in \cite{CRR,Mea,CF}). 
A proof of the result may be found in \cite[Proposition~0.9.1]{Cain}, which in turn makes use of \cite[Construction~12.3]{CP}. 

\begin{lem}\label{lem:mon-gr}
If the monoid $M=\Mon\pre{A}{\mathfrak{R}}$ is group-embeddable, then it embeds into the group with the same presentation, $G=\Gp\pre{A}{\mathfrak{R}}$.
\end{lem}

\begin{proof}[Proof of Proposition \ref{pro:mon-gr}]
($\Ra$) Assume $M$ embeds into a finitely presented group $G$. Then $G$ is also finitely presented as a monoid. Now the conclusion follows by Murski\u{\i}'s semigroup/monoid
version of the Higman Embedding Theorem \cite{Mur}.

($\La$) Suppose $M=\Mon\pre{A}{u_i=v_i\; (i\in I)}$ is \gem and $\{(u_i,v_i): i \in I\}$ is a r.e.\ subset of $A^*\times A^*$. By the previous lemma, $M$ embeds
into the group $G$ with the same presentation, which is the same as $\Gp\pre{A}{u_iv_i^{-1}=1\; (i\in I)}$ where $\{u_iv_i^{-1}:\ i\in I\}$ is now a r.e.\ language
over $\ol{A}$. Then, by the Higman Embedding Theorem, $G$ in turn embeds into a finitely presented group, and hence so does $M$.
\end{proof}

\begin{cor}\label{cor:prefix-Hig}
Any prefix monoid is a \gem recursively presented monoid.
\end{cor}

The previous corollary immediately raises the question: Which \gem recursively presented monoids arise as prefix monoids? Possibly all of them? We refute this by quickly looking at
the special case when the prefix monoid (of a finitely presented group) is itself a group.

\begin{lem}\label{lem:pr-group}
If a group arises as a prefix monoid then it is finitely presented.
\end{lem}

\begin{proof}
Assume that $G=\Gp\pre{A}{w_i=1\; (i\in I)}$ is a finitely presented group whose prefix monoid $P$ is also a group. Let $a\in\ol{A}$ be any letter occurring in some $w_i$, so that
$w_i=w'aw''$. Then both $[w']_G$ and $[w'a]_G$ are values of prefixes of $w_i$ in $G$, so they belong to the generating set of $P$. However, the assumption that $P$
is a group implies that $[w']_G^{-1}\in P$, and so 
$$[a]_G=[w']_G^{-1}[w'a]_G \in P,$$ 
as well as $[a]_G^{-1}\in P$. We conclude that $P$ coincides with the subgroup of $G$ generated by the subset $B\subseteq A$ of all letters from $A$ appearing in 
the relators $w_i$, $i\in I$. But since 
$G = \Gp\pre{B}{w_i=1\; (i\in I)} \ast FG(A \setminus B)$ 
it is then straightforward to see that $P=\Gp\pre{B}{w_i=1\; (i\in I)}$, so $P$ is finitely presented.
\end{proof}

In fact, the converse is true as well, but we show this as part of the following more general result for group-embeddable monoids.

\begin{pro}\label{pro:fin-pr}
Every \gem finitely presented monoid arises as a prefix monoid.
\end{pro}

\begin{proof}
Let $M=\Mon\pre{A}{u_i=v_i\; (i\in I)}$ be a \gem finitely presented monoid. By Lemma \ref{lem:mon-gr}, $M$ embeds into the group $G=\Gp\pre{A}{u_i=v_i\; (i\in I)}$ 
via the homomorphism induced by the identity map on $A$. 
However, the following is then also a presentation for $G$:
$$
\Gp\pre{A}{u_iv_i^{-1}=1\; (i\in I),\ aa^{-1}=1\; (a\in A)}.
$$
Now, computing the generating set for the prefix monoid $P$ with respect to this presentation of $G$ would lead us to the prefixes of the words $u_i$, words of the form
$u_ip$, where $p$ is a prefix of $v_i^{-1}$, and the individual letters $a\in A$. In the first of these cases, however, note that if $v_i^{-1}=pq$ then $v_i=q^{-1}p^{-1}$
and $[u_ip]_G=[q^{-1}]_G$, so indeed $P$ is generated by the elements of $G$ represented by the prefixes of $u_i,v_i$ ($i\in I$) and the letters $a\in A$. Since both
$u_i$ and $v_i$ are positive words (containing no inverses of letters), we conclude that $P$ is the submonoid of $G$ generated by $\{[a]_G$:\ $a\in A\}$, which is isomorphic to $M$.
\end{proof}

Proposition~\ref{pro:fin-pr} shows that all \gem finitely presented monoids occur as prefix monoids. 
In contrast, Lemma~\ref{lem:pr-group} shows not all \gem recursively presented monoids occur as prefix monoids (for example, the non-finitely recursively presented groups).  
However, we will now see that by taking a free product of such a monoid with a free monoid of sufficient finite rank will give a prefix monoid. This leads to a complete 
characterisation of prefix monoids in the following result.    

Let $\Sigma_k^*$ denote the free monoid of rank $k$. 

\begin{thm}\label{thm:almost-all}
For every \gem recursively presented monoid $M$ there is a natural number $\mu_M$ such that $M \ast \Sigma_k^*$ is a prefix monoid if and only if $k \geq \mu_M$.
Moreover, up to isomorphism, the class of all prefix monoids is equal to
$$\
\{
M \ast \Sigma_k^* : M \ \text{is a \gem recursively presented monoid and} \ k \geq \mu_M  
\}.
$$
\end{thm}

\begin{proof}
Let us start by noting that if $G$ is a finitely presented group that embeds the finitely generated monoid $M$, then there exists a finite presentation 
$$\Gp\pre{A}{w_i=1\; (i\in I)}$$
of $G$ such that $M$ is isomorphic to its submonoid generated by some subset $B\subseteq A$. So, actually, we may identify $M=\Mon\gen{[b]_G:\ b\in B}$. 
We begin by first showing that there exists some finite set $C$ and a presentation for $G*FG_2$, the free product of $G$ and a free group of rank 2, 
such that $M*C^*$ is isomorphic to the prefix monoid for the presentation of $G*FG_2$ in question.

The presentation of $K=G*FG_2$ for which we aim to describe the prefix monoid is the following one:
$$
\Gp\pre{A,s,t}{sw_is^{-1}=1\; (i\in I),\ tbt^{-1}tb^{-1}t^{-1}=1\; (b\in B)}.
$$
The required prefix monoid $P$ is generated by the elements represented in $G$ by the words $sp$, where $p$ is a (possibly empty) prefix of $w_i$ for some 
$i\in I$, and by $t,tb,tbt^{-1}$ ($b\in B$). Let $P_1$ denote the submonoid of $P$ generated by $\{[t]_K,[tb]_K,[tbt^{-1}]_K:\ b\in B\}$, and $P_2$ be the 
submonoid generated by $\{[sp]_K:\ p\in\pref(w_i),\; i\in I\}$. We claim that:
\begin{itemize}
\item[(1)] $P_1 \cong M * \{[t]_K\}^*$;
\item[(2)] $P_2$ is a free monoid of finite rank;
\item[(3)] $P \cong P_1 * P_2$.
\end{itemize}

For (1), start by noting that the generators of $P_1$ of the second type are redundant, as $[tb]_K=[tbt^{-1}]_K[t]_K$ for all $b\in B$. 
So $P_1$ is the submonoid of $K$ generated by $[t]_K$ and $\{[tbt^{-1}]_K:\ b\in B\}$. 
Now $P_1$ is isomorphic to 
$[t]_K^{-1}P_1[t]_K$ which is the submonoid of $K$ generated by $\{[b]_G:\ b\in B\}$ and $[t]_K$, and this monoid is isomorphic to $M*\{[t]_K\}^*$ 
by the Normal Form Theorem for free products \cite[Theorem IV.1.2]{LSch} applied to the group $K$. 
Let us note at this point that all elements of $P_1$ are of the form $[tm_1\dots tm_kt^{-1}]_K$ for some $k\geq 1$ and $m_1,\dots,m_k\in B^*$.

To prove (2), we claim that $P_2$ is a free monoid on $\{[sp]_K:\ p\in\pref(w_i),\; i\in I\}$. Indeed, if
$$
[sp_1]_K\dots [sp_m]_K = [sp_1\dots sp_m]_K = [sq_1\dots sq_r]_K = [sq_1]_K\dots [sq_r]_K
$$
holds in $P_2$ (and so in $K$) then again by the Normal Form Theorem in free products we must have $m=r$ and $[p_k]_G=[q_k]_G$ for all $1\leq k\leq m$, 
so $[sp_k]_K=[sq_k]_K$ for $1 \leq k \leq m$, 
hence $P_2$ satisfies no nontrivial equality among its generators.

Finally, for (3), notice first that $P_1\cup P_2$ generates $P$. 
Also, by the observations already made in the previous two paragraphs, we have that the reduced forms (with respect to $K$) of all elements of $P_1$ are 
either of the type $[t^{u_1}m_1\dots t^{u_k}m_k]_K$ or of the type $[t^{u_1}m_1\dots t^{u_k}m_kt^{-1}]_K$ for some $u_j\geq 1$ and  words $m_1,\dots,m_k\in B^*$ 
such that $[m_j]_G\neq 1$ for all $1\leq j\leq k$. Similarly, the reduced forms of all elements of $P_2$ are of the type $[s^{v_1}p_1\dots s^{v_l}p_l]_K$ for some
$v_j\geq 1$ and words $p_1,\dots,p_l\in\ol{A}^*$ (specifically, all these words are non-empty prefixes of relator words $w_i$, $i\in I$) such that $[p_j]_G\neq 1$ 
for all $1\leq j\leq l$. Therefore, if we have an equality 
$$
[\alpha_1 \beta_1 \ldots \alpha_m \beta_m]_K = [\gamma_1 \delta_1 \ldots \gamma_n \delta_n]_K
$$
where $[\alpha_q]_K,[\gamma_r]_K\in P_1$, $[\beta_q]_K,[\delta_r]_K\in P_2$ ($1\leq q\leq m$, $1\leq r\leq n$) are all non-trivial, except possibly some of 
$[\alpha_1]_K,[\beta_m]_K,[\gamma_1]_K,[\delta_n]_K$, and the words $\alpha_q, \beta_q, \gamma_q$ and $\delta_q$ are all words in the reduced forms described above for $P_1$ and $P_2$. 

Since every word $\alpha_q,\gamma_r$ begins with $t$ and ends with either a letter from $B$ or $t^{-1}$, and every word $\beta_q,\delta_r$ begins with $s$ and ends with 
a letter from $A$, it follows that both the words $\alpha_1 \beta_1 \ldots \alpha_m \beta_m$ and $\gamma_1 \delta_1 \ldots \gamma_n \delta_n$ are already in reduced form 
with respect to the free product $K$. This yields an equality of reduced forms in $K$ of the following type:
$$
[x_1y_1\dots x_\sigma y_\sigma]_K = [x'_1y'_1\dots x'_\tau y'_\tau]_K,
$$
where 
$\alpha_1 \beta_1 \ldots \alpha_m \beta_m = x_1y_1\dots x_\sigma y_\sigma$ and $\gamma_1 \delta_1 \ldots \gamma_n \delta_n = x'_1y'_1\dots x'_\tau y'_\tau$ and 
each $x_j,x_j'$ is either a power of $t$, or a power of $s$, or of the form $t^{-1}s^v$ for some $v\geq 1$ and each $y_j,y_j'$ is a word over $\ol{A}$ representing 
a non-trivial element of the group $G$ (in case of words occurring immediately after powers of $t$, there are in fact words from $B^*$). By employing the Normal Form Theorem 
in free products once again, we conclude that $\sigma=\tau$ and that $x_j=x'_j$ and $[y_j]_G=[y'_j]_G$ for all $1\leq j\leq \sigma$. But then combining this with the equalities 
of words $\alpha_1 \beta_1 \ldots \alpha_m \beta_m = x_1y_1\dots x_\sigma y_\sigma$ and $\gamma_1 \delta_1 \ldots \gamma_n \delta_n = x'_1y'_1\dots x'_\tau y'_\tau$ it follows 
that $m=n$, $[\alpha_q]_K=[\gamma_q]_K$ and $[\beta_q]_K=[\delta_q]_K$ for all $1\leq q\leq m$. This suffices to establish that $P$ is a free product of $P_1$ and $P_2$.

For the remainder of the proof, assume that, for some $k$, $M*\Sigma_k^*$ is a prefix monoid in the finitely presented group $L=\Gp\pre{X}{\mathfrak{R}}$. Let $y$ be a symbol 
not in $X$, and consider the group
$$
L' = \Gp\pre{X,y}{\mathfrak{R},\; yy^{-1}=1}.
$$
We have that $L'\cong L*FG(y)$. Then the prefix monoid for the given presentation of $L'$ is generated by the prefix monoid of the considered presentation of $L$ and the element 
represented by the letter $y$. Therefore, the prefix monoid of $L'$ is isomorphic to $M * \Sigma_k^* * \{y\}^* \cong M*\Sigma_{k+1}^*$. Hence, if $\mu_M$ is the minimal value of $k$ 
such that $M*\Sigma_k^*$ arises as a prefix monoid (of a finitely presented group), it follows that $M*\Sigma_k^*$ is a prefix monoid for all larger values of $k$. From this remark, 
the statement of the first sentence in the theorem follows. 

Conversely, given any \gem recursively presented monoid $M$ that arises as a prefix monoid we have $\mu_M=0$ and so $M \cong M \ast \Sigma_0^*$ belongs to the class in the 
statement of theorem.  
\end{proof}

\begin{rmk}
The size of the finite set $C$ appearing in the first part of the proof of the previous result is bounded from above by one plus the sum of lengths of relators needed to define 
the finitely presented group $G$ so that the generating set of $M$ is included into the generating set of $G$. More precisely, $|C|-1$ is actually the number of different elements 
of the group $G$ represented by all prefixes of relators.
\end{rmk}


\section{Groups of units and \Sch groups of prefix monoids}

Bearing in mind the goals of this section, the characterisation of the classes of groups arising as groups of units and \Sch groups in prefix monoids, respectively, we first record 
several basic facts about \gem monoids. Actually, the scope of some of these auxiliary results extend to the wider classes of (right, left) cancellative monoids.

\begin{pro}\label{pro:general}
Let $M$ be a monoid with $x, y \in M$ and let $U$ be its group of units.
\begin{itemize}
\item[(1)] If $M$ is right cancellative then $x\,\L\,y$ if and only if $y=ux$ for some $u\in U$. Consequently, $L_x=Ux$.
\item[(2)] If $M$ is left cancellative then $x\,\R\,y$ if and only if $y=xu$ for some $u\in U$. Consequently, $R_x=xU$.
\item[(3)] If $M$ is cancellative then $H_x=Ux\cap xU$ and $D_x=UxU$.
\item[(4)] 
If $M$ is a submonoid of a group $G$ then the \Sch group of $H_x$ is isomorphic to $U\cap x^{-1}Ux$. 
\end{itemize}
\end{pro}

\begin{proof}
(1) $x\,\L\,y$ holds if and only if $y=ux$ and $x=vy$ for some $u,v\in M$. But then $x=vux$, and by right cancellativity of $M$ it follows that $vu=1$. Similarly, $y=uvy$ 
implies $uv=1$, so both $u,v$ must be units of $M$. (2) is dual to (1), and (3) follows from (1) and (2). 

Assume now that $M$ is a submonoid of a group $G$; we compute the right \Sch group of $H_x$. Indeed, $m\in M$ stabilises $H_x$ from the right if and only if $H_xm\subseteq H_x$.
In particular, $xm\,\R\,x$, which by (2) implies $m\in U$. Furthermore, we must have $xm\,\L\,x$, so $xm=ux$ for some $u\in U$. This, in $G$, yields $m=x^{-1}ux\in x^{-1}Ux$.
Conversely, the assumption $m\in U\cap x^{-1}Ux$ implies that $xm\in xU\cap xx^{-1}Ux=xU\cap Ux=H_x$.
Since $xm \in H_x$ it follows by Green's Lemma \cite{How} that $H_x m=H_x$.  
So, $\Stab(H_x)=U\cap x^{-1}Ux$. However, the (right) \Sch group of $H_x$ must coincide now with its right stabiliser, as left cancellativity of $M$ implies that its every
element must induce a different permutation of $H_x$. 
\end{proof}

\begin{pro}\label{pro:general2}
Let $M$ be a right (or left) cancellative monoid and let $U$ be its group of units.
\begin{itemize}
\item[(1)] $M\setminus U$ is an ideal of $M$. Consequently, if $M$ is finitely generated (resp.\ recursively presented), so is $U$.
\item[(2)] Every \Sch group of $M$ embeds into $U$.
\item[(3)] If $M$ is recursively presented then every \Sch group of $M$ is a recursively enumerable subgroup of a finitely presented group. 
\end{itemize}
\end{pro}

\begin{proof}
We consider only the case when $M$ is right cancellative, the left cancellative case being dual.

(1) Assume that $x,y\in M$ are such that $xy\in U$. Then $xyz=zxy=1$ for some $z\in M$. Hence, $yzxy=y$ which by right cancellativity implies $yzx=1$. So, $zx$ is an inverse 
of $y$ and $yz$ is an inverse of $x$, thus $x,y\in U$. Therefore, if any of $a,b\in M$ does not belong to $U$ it follows that $ab\in M\setminus U$, showing that $M\setminus U$ 
is an ideal of $M$.

Now if $X\subseteq M$ is any generating set of $M$ it follows that $X\cap U$ must be a (monoid) generating set of $U$: indeed, the previous paragraph shows that if for $u\in U$ 
we have $u=x_1\dots x_n$ for some $x_1,\dots,x_n\in X$ then in fact we must have $x_1,\dots,x_n\in X\cap U$. So, if $X$ is finite (witnessing that $M$ is finitely generated), 
so is $X\cap U$, thus $U$ is a finitely generated group.

Finally, assume that $M$ is recursively presented. As just shown, this implies that the group $U$ is finitely generated, and it embeds, together with the whole $M$, into a 
finitely presented monoid \cite{Mur}. So, as a monoid, $U$ is recursively presented, say $U=\Mon\pre{A}{u_i=v_i\; (i\in I)}$. Since $U$ is a group, the group presentation 
$\Gp\pre{A}{u_iv_i^{-1}=1\; (i\in I)}$ also defines $U$. Here, $\{u_iv_i^{-1}:\ i\in I\}$ is a r.e.\ language, showing $U$ to be a recursively presented group.
 
(2) Here we use the fact (already mentioned in the preliminary section) that the (right) \Sch group of an $\H$-class is isomorphic to the left \Sch group of that $\H$-class. 
Since $M$ is right cancellative, if $H$ is an $\H$-class of $M$, the condition $uH=H$ for some $u\in M$, implies, by Proposition \ref{pro:general} (1), that $u$ is a unit of $M$. 
Furthermore, if $u'\neq u$ is also a unit of $M$ then for any $h\in H$ we have $uh\neq u'h$ because of the right cancellative property. Hence, every element $u$ of the left stabiliser 
of $H$ induces (via the left translation $\lambda_u:h\mapsto uh$, $h\in H$) a different permutation of $H$, implying that it coincides with the left \Sch group of $H$ and embeds into $U$.

(3) Let $H$ be an $\H$-class of $M=\Mon\pre{A}{u_i=v_i\; (i\in I)}$. By (2), the \Sch group $K$ of $H$ embeds into $U$. By (1), the group of units $U$ is recursively presented, say 
$U=\Gp\pre{A\cap U}{\mathfrak{R}}$, and so it embeds into a finitely presented group. By combining these two facts, we obtain that $K$ is a subgroup of a finitely presented group 
$G=\Gp\pre{B}{\mathfrak{R}'}$; in this sense we aim to show that $K$ is a recursively enumerable subgroup of a finitely presented group. 

Assume that $H=H_x$ for some $x\in M$, and let $w\in A^*$ be any word such that $[w]_M=x$. Also, since $U$ embeds into $G$, there is a mapping which assigns to each letter 
$a\in A\cap U$ a word $w_a\in\ol{B}^*$ inducing the embedding $U\to G$ from the previous paragraph, so that $[a_1\dots a_m]_U$ is mapped to $[w_{a_1}\dots w_{a_m}]_G$ for all 
$a_1,\dots,a_m\in A\cap U$. Now note that for $u\in A^*$, $[u]_M\in K$ if and only if $[u]_M$ is a unit of $M$ and there exist words $p,q\in A^*$ such that  $[uw]_M=[wp]_M$ and 
$[wpq]_M=[w]_M$. The latter two equalities ensure that $[u]_M x\,\R\,x$ and so $[u]_M x\,\H\,x$ (as we have $[u]_M x\,\L\,x$ for granted, by Proposition \ref{pro:general} (1)), 
and by Green's Lemma (see e.g.\ \cite[Lemma 2.2.1 and its dual, Lemma 2.2.2]{How}) the condition $[u]_M x\in H_x=H$ is equivalent to $[u]_M H=H$.

Therefore, (3) will be proved as soon as we construct a Turing machine $\mathcal{M}$ that halts on input $u\in A^*$ if and only if the word $u$ has the properties described in 
the previous paragraph, for then the language
$$
\{w_{a_1}\dots w_{a_m}:\ a_1\dots a_m\in L(\mathcal{M})\}\subseteq\ol{B}^*,
$$
whose words represent the elements of the image of $K$ under the considered embedding $U\to G$, is clearly r.e. In more detail, $\mathcal{M}$ should ``detect'' the words $u\in A^*$ 
such that the equalities 
$$
[uv]_M=[vu]_M=1, \quad [uw]_M=[wp]_M, \quad [wpq]_M=[w]_M
$$
hold for some $v,p,q\in A^*$.

Since $M$ is a recursively presented monoid, its word problem, that is, the set $\{(u,v)\in A^*\times A^*:\ [u]_M=[v]_M\}$, is well-known to be r.e. (For groups this is already contained 
in the results of Higman \cite{Hig}, and for monoids this is a consequence of the work of Murski\u{\i} \cite{Mur}.) Hence, there exists a Turing machine $\mathcal{M}_M$ whose language 
is the word problem of $M$. Use some of the standard enumerations of triples of all words over $\ol{A}$ to obtain $\ol{A}^*\times\ol{A}^*\times\ol{A}^*=\{(v_n,p_n,q_n):\ n\in\N\}$.  
Then, use countably infinitely many copies $\mathcal{M}_n$, $n\in\N$, of the machine $\mathcal{M}_M$ and feed the inputs $(uv_i,1)$, $(v_iu,1)$, $(uw,wp_i)$ and $(wp_iq_i,w)$ into 
the machines $\mathcal{M}_{4i}$, $\mathcal{M}_{4i+1}$, $\mathcal{M}_{4i+2}$, $\mathcal{M}_{4i+3}$, $i\in\N$, respectively, so that the resulting machine first performs the first step 
of machine $\mathcal{M}_0$, then the first step of machine $\mathcal{M}_1$ and the second step of machine $\mathcal{M}_0$, and so on. Such a machine $\mathcal{M}$ will halt and return 
a positive answer if and only if the word $u$ represents (in $M$) an element of $K$, the required \Sch group of $H$. 
\end{proof}

\begin{thm}\label{thm:units}
A group arises as a group of units of a prefix monoid if and only if it is a recursively presented group.
\end{thm}

\begin{proof}
By (1) of the previous proposition, it is immediate that the group of units of any prefix monoid is recursively presented.

Conversely, let $G=\Gp\pre{A}{\mathfrak{R}}$ be a recursively presented group. Then $G=\Mon\pre{A\cup A'}{\mathfrak{R}',\; aa'=a'a=1\; (a\in A)}$ is a monoid presentation for $G$, where 
$A'$ is a bijective copy of $A$ disjoint form $A$ and the relations from $\mathfrak{R}'$ are obtained from $\mathfrak{R}$ by replacing each $a^{-1}$ by $a'$, $a\in A$. This shows that $G$, 
as a monoid, is recursively presented. By Theorem \ref{thm:almost-all}, there is an integer $k\geq 0$ such that $G*\Sigma_k^*$ is a prefix monoid. It remains to note that $G$ is the group of units 
of $G*\Sigma_k^*$.
\end{proof}

The main technical tool for studying \Sch groups in various classes of \gem finitely generated monoids is the following.  

\begin{thm}\label{thm:wHig}
Let $\cK$ be a class of \gem recursively presented monoids such that for any \gem recursively presented monoid $M$ and any embedding $\alpha:M\to \Gamma$ into a finitely presented group $\Gamma$ there 
exist $N_M\in\cK$, a finitely presented group $\Gamma_1$, monoid embeddings $\beta:M\to N_M$ and $\alpha':N_M\to \Gamma_1$, and a group embedding $\beta':\Gamma\to\Gamma_1$ 
such that the following diagram commutes:
$$
\xymatrix@C+2pc{M \ar[r]^{\beta} \ar[d]_{\alpha} & N_M \ar[d]^{\alpha'} \\
\Gamma \ar[r]_{\beta'} & \Gamma_1}
$$
and that the restriction $\beta\restriction_{U_M}$ to the group of units of $M$ is an isomorphism between $U_M$ and $U_{N_M}$, the group of units of $N_M$.

Then for any group $K$ that is a recursively enumerable subgroup of a finitely presented group there is a monoid $S\in\cK$ such that $K$ is isomorphic to the \Sch group of an $\H$-class of $S$.
\end{thm}

\begin{proof}
Let $H_1=\Gp\pre{A}{\mathfrak{R}}$ be a finitely presented group such that $K$ embeds into $H_1$ so that the image of this embedding, $K_1=\{[w]_{H_1}:\ w\in L\}$ (for some r.e.\ language $L$), is a recursively enumerable subgroup of $H_1$. 
By \cite[Lemma 5.2]{GK} there is a finitely presented group containing two conjugate subgroups, both isomorphic to $H_1$, such that their intersection is isomorphic to $K_1\cong K$. 
In more detail, let $G_0$ be the HNN extension of $H_1$ with a stable letter $t$, conjugating each element of $K_1$ to itself. 
A presentation for this group is $\Gp\pre{A,t}{\mathfrak{R}\cup\mathfrak{R}_0}$, where $\mathfrak{R}_0=\{t^{-1}wtw^{-1}:\ w\in L\}$. Here $\mathfrak{R}_0$ is clearly a r.e.\ language, so $G_0$ is finitely generated and recursively presented. 
By the Higman Embedding Theorem, it embeds into some finitely presented group $G$.
Within $G$ we have $H_2=[t]_G^{-1} H_1 [t]_G \cong H_1$ and $H_1\cap H_2=K_1$.

Now let $M$ be the submonoid of $G$ generated by $H_1$ and $[t]_G$. This is a finitely generated submonoid of $G$, a finitely presented group, so $M$ is a \gem recursively presented monoid.
A generating set of $M$ is $\{[a]_G,[a]_G^{-1}:\ a\in A\}\cup\{[t]_G\}$. 
Let us pause for a moment to analyse the group of units of $M$. 
Indeed, for a word $w\in(\ol{A}\cup\{t\})^*$ we have that $[w]_G$ is invertible in $M$ if and only if $[w]_G,[w]_G^{-1}\in M$. Let $w$ be one such word. Then we can write 
$$
[w]_G = [h_0th_1t\dots th_\ell]_G
$$
for some $\ell\geq 0$ and words $h_i\in\ol{A}^*$, $0\leq i\leq\ell$, such that $[h_i]_G\in H_1$.
Note that the word $h_0th_1t\dots th_\ell$ is already in a reduced form with respect to the HNN-extension structure of $G_0$ (see \cite[p.181]{LSch}).
Hence,
$$
[w]_G^{-1} = [h_\ell^{-1}t^{-1}\dots t^{-1}h_1^{-1}t^{-1}h_0^{-1}]_G.
$$
The word on the right hand side is also already in a reduced form. However, the condition $[w]_G^{-1}\in M$ and the very definition of $M$ implies that there must be
words $h_0',h_1',\dots,h_p'\in\ol{A}^*$ with $[h_j']_G\in H_1$ for all $0\leq j\leq p$ such that 
$$
[h_\ell^{-1}t^{-1}\dots t^{-1}h_1^{-1}t^{-1}h_0^{-1}]_G = [h_0'th_1't\dots th_p']_G.
$$
This is now an equality of two reduced forms in the HNN-extension $G_0$ and by \cite[Lemma 6.2]{DG} it is impossible unless $\ell=p=0$ and $[h_0']_G=[h_0]_G^{-1}\in H_1$.
On the other hand, it is clear that any element of $H_1$ is a unit in $M$, so $U_M=H_1$.

Next, start with $M$ and its inclusion map into $G$. By the condition of the theorem, there is a monoid $S=N_M\in\cK$, a finitely presented group $G_1$ and embeddings 
$\phi,\psi,\xi$ such that the diagram 
$$
\xymatrix@C+2pc{M \ar[r]^{\phi} \ar[d]_{\subseteq} & S \ar[d]^{\psi} \\
G \ar[r]_{\xi} & G_1}
$$
commutes, and the corresponding restriction of $\phi$ induces an isomorphism between the group of units $U_M$ and $U_S$. By Proposition \ref{pro:general} (4), the \Sch group of
$H_{[t]_G}$ is isomorphic to $U_M\cap [t]_G^{-1}U_M[t]_G = H_1\cap H_2\cong K$. Applying Proposition \ref{pro:general} (4) to $H_{\mathbf{t}}$ in $S$, where $\mathbf{t}=[t]_G\phi$, the \Sch group of $H_{\mathbf{t}}$ is isomorphic to 
$$
U_S\psi\cap (\mathbf{t}^{-1}U_S\mathbf{t})\psi.
$$
However, since $U_S=U_M\phi$ and $\phi\psi=\xi\restriction_M$, the latter group is just $K_1\xi$, an isomorphic copy of $K$.
\end{proof}

\begin{thm}\label{thm:Sch-prefix}
The \Sch groups of prefix monoids are exactly the recursively enumerable subgroups of finitely presented groups.  
\end{thm}

\begin{proof}
($\Ra$) 
By Corollary \ref{cor:prefix-Hig}, if $M$ is a prefix monoid then it is \gem (and so right cancellative) and recursively presented. Hence by Proposition \ref{pro:general2} (3) every \Sch group of $M$ is a recursively enumerable subgroup of a finitely presented group.

($\La$) 
In Theorem \ref{thm:wHig}, set $\cK$ to be the class of prefix monoids. The first part of the proof of Theorem \ref{thm:almost-all} shows that for any \gem recursively presented monoid $M$ and any embedding $M\to G$ 
into a finitely presented group $G$, there exists a finite set $C$ such that $M*C^*$ is isomorphic to a prefix monoid in $G*FG_2$. Since the group of units of $M*C^*$ is 
the same as that of $M$, the natural embedding $M\to M*C^*$, the isomorphism between $M*C^*$ and a prefix monoid of $G_1=G*FG_2$ as in the proof of Theorem \ref{thm:almost-all}, 
and the group embedding $G\to G_1$ defined by 
$$
g\mapsto [t]_{G_1}g[t]_{G_1}^{-1}
$$
for all $g\in G$, provide us with all the required parameters for the application of Theorem \ref{thm:wHig}. It now yields that any recursively enumerable subgroup of a finitely presented group occurs as a \Sch group of a prefix monoid.
\end{proof}


\section{On RU-monoids}

In the previous two sections we studied prefix monoids, giving a characterisation of those monoids that arise as prefix monoids, and of their groups of units and \Sch groups. 
In this section we start by collecting a few basic facts about the submonoids of right units of finitely presented special inverse monoids which, as mentioned above, we call the RU-monoids;
subsequently, we provide a characterisation of groups arising as \Sch groups of RU-monoids. While prefix monoids are all group-embeddable, for RU-monoids we need to work with the 
wider class of right cancellative monoids.  

\begin{lem}\label{lem:ru-hig}
Every RU-monoid is a right cancellative recursively presented monoid.
\end{lem}

\begin{proof}
First of all, every RU-monoid $R$ is finitely generated, namely by the elements represented by all prefixes of the relator words appearing in the presentation of the finitely 
presented special inverse monoid $I=\Inv\pre{A}{\mathfrak{R}}$, the monoid of right units of which is isomorphic to $R$. However, as a monoid, $I\cong \ol{A}^*/\theta_{\mathfrak{R}}$, 
where $\theta_{\mathfrak{R}}$ is the smallest congruence of the free monoid $\ol{A}^*$ containing $\mathfrak{R}$ and the Wagner congruence. We conclude that $I$ is recursively 
presented as a monoid, and thus it embeds (together with $R$) into a finitely presented monoid $M$. Therefore, $R$ is a recursively presented monoid.
\end{proof}

Analogously to the case of prefix monoids, the first question that arises following the previous lemma
is if all right cancellative recursively presented monoids arise as RU-monoids. (Shortly we shall 
see that there are two approaches to defining right cancellative monoids by means of presentations.) The answer is negative: namely, there is a parallel result to that of Lemma \ref{lem:pr-group}.

\begin{lem}\label{lem:ru-group}
If the monoid $R$ of right units of a finitely presented special inverse monoid $M=\Inv\pre{A}{w_i=1\; (i\in I)}$ is a group, then it is finitely presented.
\end{lem}

\begin{proof}
First of all, recall that $R$ is the submonoid of $M$ generated by all of its elements of the form $[p]_M$ where $p$ is a prefix of a word $w_i$, $i\in I$. Let $a\in\ol{A}$ be any letter occurring in some $w_i$, so that $w_i=w'aw''$. Similarly as in the proof of Lemma \ref{lem:pr-group}, we have that  both $[w']_M$ and $[w'a]_M$ are values of prefixes of $w_i$ in $M$, so they belong to $R$. Since $R$ is a group, $[w']_M^{-1}\in R$. Therefore, $[w']_M^{-1}$ is a right unit of $M$, so $[w']_M^{-1}[w']_M=1$. This allows us to conclude that 
$$[a]_M=[w']_M^{-1}[w'a]_M \in R,$$
and so $[a]_M$ is a unit of $M$ for any letter $a\in A$ that appears in some $w_i$. 
If $A'\subseteq A$ denotes the set of all such letters, we obtain that the special inverse monoid $\Inv\pre{A'}{w_i=1\; (i\in I)}$ is a group and it coincides with the group with the same presentation, 
$G=\Gp\pre{A'}{w_i=1\; (i\in I)}$. Now $M\cong G * FIM(A\setminus A')$, where this free product is taken in the category of inverse monoids, from which it instantly follows that $R\cong G$, 
so $R$ is a finitely presented  group.
\end{proof}

The converse of this lemma is also true, as any finitely presented group is finitely presented as a special inverse monoid (by adding in the relations $aa^{-1} = 1$ and $a^{-1}a=1$ for all generators $a$) and, viewed this way, is the monoid of right units of itself. 

Now, while dealing with right cancellative monoids, there is yet another type of presentation that is worthwhile considering in this context. Namely, since right cancellative
monoids form a quasi-variety (defined by the implication $xz=yz \Ra x=y$), we have that for any monoid $M$ and any family $\{\rho_i:\ i\in I\}$ of its congruences
with the property that $M/\rho_i$ is right cancellative, the intersection $\rho$ of this family is again a congruence of $M$ such that $M/\rho$ is right cancellative.
For this reason, not unlike the greatest group image (of an inverse monoid, for example), there exists the greatest right cancellative image of a monoid. In particular,
if $M=\Mon\pre{A}{\mathfrak{R}}$, this greatest right cancellative image is denoted by $\MonRC\pre{A}{\mathfrak{R}}$, and it is defined as $A^*/\mathfrak{R}^{\mathrm{RC}}$, 
where $\mathfrak{R}^{\mathrm{RC}}$ is the intersection of all congruences $\sigma$ of $A^*$ with the property that $\sigma\supseteq\mathfrak{R}$ and $A^*/\sigma$ is 
right cancellative. So, a monoid $M=\Mon\pre{A}{u_i=v_i\; (i\in I)}$ is right cancellative (and there are conditions implying this, see e.g.\ \cite{Adj})
if and only if $\Mon\pre{A}{u_i=v_i\; (i\in I)}=\MonRC\pre{A}{u_i=v_i\; (i\in I)}$.

When both $A,\mathfrak{R}$ are finite, $M=\MonRC\pre{A}{\mathfrak{R}}$ is said to be \emph{finitely RC-presented}. The previous remark implies that when the finitely presented 
monoid $\Mon\pre{A}{\mathfrak{R}}$ happens to be right cancellative, it is also finitely RC-presented; the converse of this, however, is not true in general (cf.\ \cite{CRR}
and the references mentioned there). So, finitely RC-presented (right cancellative) monoids form a strictly wider class than that of finitely presented monoids which are
right cancellative. By the following result we include the former class into the scope of our considerations.

\begin{thm}\label{thm:rc-fin}
Every finitely RC-presented monoid $T=\MonRC\pre{A}{u_i=v_i\; (i\in I)}$ is an RU-monoid.
\end{thm}

\begin{proof}
Begin by defining a finitely presented special inverse monoid
$$
M_T = \Inv\pre{A}{aa^{-1}=1\; (a\in A),\ u_iv_i^{-1}=1\; (i\in I)}.
$$
We claim that its monoid of right units is isomorphic to $T$. Notice that the inverse monoid $M_T$ is also presented by $\Inv\pre{A}{aa^{-1}=1\; (a\in A),\ u_i=v_i\; (i\in I)}$: 
under the assumption that all the letters $a\in A$ represent right units, the equations $u_iv_i^{-1}=1$ and $u_i=v_i$ are equivalent, by using that fact that if a word 
$w\in\ol{A}^*$ and an inverse monoid $K$ is such that $[w]_K$ is a right unit of $K$ then $[w]_K=[\red(w)]_K$, see e.g.\ \cite[Corollary 3.2]{Gr-Inv}.

We proceed by mimicking the argument from \cite[Theorem 2.2]{IMM}. Namely, $T$ is a right cancellative monoid, so \cite[Theorem 1.22]{CP} applies: $T$ is isomorphic to the monoid
of right units of its \emph{inverse hull} $IH(T)$, the inverse monoid of partial bijections on $T$ generated by the right translations $\rho_t$, $t\in T$ (since $T$ is right cancellative,
each $\rho_t$ must be an injective transformation of $T$, with the left ideal $Tt$ as its image). Since for all $i\in I$, both $[u_i]_{IH(T)}$ and $[v_i]_{IH(T)}$ are right units of
the inverse monoid $IH(T)$, we have $[u_iv_i^{-1}]_{IH(T)}=1$ in this inverse monoid. Hence, there is a surjective inverse monoid homomorphism $\mu:M_T\to IH(T)$ extending the map
$[a]_{M_T}\mapsto \rho_a$, $a\in A$; also, there is a natural surjective involutory monoid homomorphism $\nu:\ol{A}^*\to M_T$. Write $T'=A^*\nu$, which is a submonoid of $M_T$
generated by $[a]_{M_T}$ for all $a\in A$. This is precisely the generating set for the monoid of its right units of $M_T$, so this monoid must be $T'$. However, $T'\mu$ is just
the monoid of right units of $IH(T)$ (the latter being isomorphic to $T$). But we have $[u_i]_{M_T}=[v_i]_{M_T}$ for all $i\in I$, and since $u_i,v_i\in A^*$ this implies
$[u_i]_{T'}=[v_i]_{T'}$. As $T'$ is right cancellative (being an RU-monoid), it is a homomorphic image of $T$, forcing $\mu\restriction_{T'}$ to be a monoid isomorphism. Thus $T$
is an RU-monoid.
\end{proof}

\begin{cor}
All \gem finitely presented monoids are RU-monoids.
\end{cor}

However, as the previous theorem shows, RU-monoids are not confined to \gem monoids; for example, the right units of 
$$
M=\Inv\pre{a,b,c}{aa^{-1}=bb^{-1}=cc^{-1}=1,\ abc^{-1}a^{-1}=1}
$$
is a right cancellative monoid that is not group-embeddable. Indeed, the proof above tells us that in fact $M=\Inv\pre{a,b,c}{aa^{-1}=bb^{-1}=cc^{-1}=1,\ ab=ac}$.
Now, in the monoid $R$ of right units of $M$ we clearly must have $[ab]_M=[ac]_M$, but certainly not $[b]_M=[c]_M$ (as, in fact, $R=\MonRC\pre{a,b,c}{ab=ac}$), 
which contradicts group-embeddability of $R$.

Since every unit of an inverse monoid is automatically a right unit, it follows that the group of units of the RU-monoid of an inverse monoid $M$ coincides with the group of units of 
of $M$; therefore, as mentioned in the introduction, \cite[Theorem 4.1]{GK} shows that the class of groups of units of all RU-monoids is precisely the class of recursively presented
groups. 

In the final result of this paper we describe the \Sch groups of RU-monoids. For this, we recall a construction from \cite[Section~6]{GR} and a couple of pertinent results from that paper that we are going to use in the following.
Let $A=\{a_1,\dots,a_n\}$ be a finite alphabet, and let $Q$ and $W$ be two subsets of $\ol{A}^*$ with $W=\{w_1,\dots,w_k\}$ finite and $Q=\{r_i:\ i\in I\}$ where $I$ contains a distinguished index 1.
Furthermore, let $t$ be a letter not in $\ol{A}$. Let $K_Q=\Gp\pre{A}{r_i=1\; (i\in I)}$. For any list of words $u_1,\dots,u_m\in\ol{A}^*$ we define
$$
e(u_1,u_2,\dots,u_m) = u_1u_1^{-1}u_2u_2^{-1}\dots u_mu_m^{-1}.
$$
Now we define the special inverse monoid
$$
M_{Q,W} = \Inv\pre{A,t}{fr_1=1,\; r_i=1\; (i\in I\setminus\{1\})}
$$
where
$$
f = e(a_1,\dots,a_n,tw_1t^{-1},\dots,tw_kt^{-1},a_1^{-1},\dots,a_n^{-1}).
$$
As noted in \cite{GR}, it was shown in \cite[Lemma 3.3]{Gr-Inv} that $M_{Q,W}$ is equal to the  
monoid defined by the presentation with generating set $A \cup \{t\}$ and defining relations  
\begin{align*}
& r_i=1 & (i\in I), \\
& aa^{-1}=a^{-1}a=1 & (a\in A), \\
& tw_jt^{-1}tw_j^{-1}t^{-1} = 1 & (1\leq j\leq k). 
\end{align*}
We are going to make use of the following results from \cite{GR}. 

\begin{pro}[{\cite[Theorem 6.3]{GR}}]\label{pro:gr-import}
With the above notation and definitions we have the following. 
\begin{itemize}
\item[(1)] Let $T_W$ be the submonoid of the group $K_Q$ generated by 
$\{[w_j]_{K_Q}:\ 1\leq j\leq k\}$. The submonoid of 
right units of $M_{Q,W}$ is isomorphic to the submonoid of the group $G_Q=K_Q * FG(t)$ generated by 
$$
\{[t]_{G_Q}\}\cup K_Q\cup [t]_{G_Q}T_W[t]_{G_Q}^{-1}.
$$
\item[(2)] The group of units of $M_{Q,W}$ is isomorphic to the subgroup $U$ of $G_Q$ generated by 
$$K_Q\cup \{[twt^{-1}]_{G_Q}:\ w\in W'\},$$
where $W'\subseteq W$ consists of all $w_j\in W$ such that the word $tw_jt^{-1}$ represents a unit of $M_{Q,W}$. 
Furthermore, $U$ is generated by its subgroups $K_Q$ and $[t]_{G_Q}H_{W'}[t]_{G_Q}^{-1}$ where $H_{W'}$ is the subgroup of $K_Q$ generated by $\{[w]_{K_Q}:\ w\in W'\}$, and is in fact isomorphic to the free product of these subgroups.
\end{itemize}
\end{pro}

Item (1) of the previous proposition is precisely Theorem 6.3 (iii) from \cite{GR}, while item (2) 
may be found within the proof of part (v) of the same theorem.

\begin{thm}\label{thm:Sch-ru}
The \Sch groups of RU-monoids are exactly the recursively enumerable subgroups of finitely presented groups. 
\end{thm}

\begin{proof}
($\Ra$) 
By Lemma \ref{lem:ru-hig}, if $M$ is an RU-monoid then it is right cancellative and recursively presented (as a monoid). Hence by Proposition \ref{pro:general2} (3) 
every \Sch group of $M$ is a recursively enumerable subgroup of a finitely presented group.

($\La$)
The initial setup for our argument is the same as in the beginning of the proof of Theorem \ref{thm:wHig}. Namely, let $K$ be an arbitrary group that arises as a recursively enumerable 
subgroup of a finitely presented group. Let $H_1=\Gp\pre{B}{\mathfrak{R}_1}$ be a finitely presented group and $L$ a r.e.\ language such that $K_1=\{[w]_{H_1}:\ w\in L\}$ is a subgroup of 
$H_1$ isomorphic to $K$. We let $G_0$ to be the HNN extension of $H_1$, this time with $z$ as the stable letter (the letter $t$ will be reserved for a different purpose in this proof), 
conjugating each element of $K_1$ to itself. As seen in Theorem \ref{thm:wHig}, $G_0$ is a recursively presented group, so it embeds into a finitely presented group $G=\Gp\pre{A}{\mathfrak{R}}$.
Here we may assume that $B\subseteq A$ and $z\in A\setminus B$, so that the subgroup of $G$ generated by $\{[b]_G:\ b\in B\}$ is equal to the embedded copy $H$ of $H_1$ in $G$. 

Bearing in mind the notation introduced just before the statement of this theorem, consider the special inverse monoid $M=M_{\mathfrak{R},W}$ where $W=B\cup B^{-1}\cup\{z\}$ (so all words from
$W$ here have length 1). By Proposition \ref{pro:gr-import} (1) we have that the monoid of right units of $M$ is isomorphic to the submonoid $R$ of the free product $P=G*FG(t)$ generated by 
$$\{[t]_P\}\cup\{[a]_P,[a]_P^{-1}:\ a\in A\}\cup [t]_P T_W [t]_P^{-1},$$
where $T_W$ is the submonoid of $G$ generated by $\{[z]_G\}\cup \{[b]_G,[b]_G^{-1}:\ b\in B\}$ (in other words, by $[z]_G$ and $H$).

Furthermore, item (2) of Proposition \ref{pro:gr-import} implies that the group of units of $M$ is isomorphic to the subgroup $U$ of $P$ generated by 
$$\{[a]_P:\ a\in A\}\cup \{[tbt^{-1}]_P:\ b\in B \},$$
with $U$ being the free product of $G$ and the subgroup $[t]_P H [t]_P^{-1}$ of $P$. This is a consequence of the fact that the word $tbt^{-1}$ represents 
a unit of $M$ for all $b\in B$, while $tzt^{-1}$ does not (even though it does represent a right unit), as we shall now justify. Indeed, 
from the displayed equations immediately before the statement of Proposition~\ref{pro:gr-import},  
the fact that $B\cup B^{-1}\subseteq W$ implies that
$[(tbt^{-1})(tb^{-1}t^{-1})]_M = [(tb^{-1}t^{-1})(tbt^{-1})]_M = 1$ holds, showing $[tbt^{-1}]_M$ and also $[tb^{-1}t^{-1}]_M$ is a unit of $M$. 
On the other hand, seeking a contradiction assume that  
$tzt^{-1}$ represents a unit of $M$. Then $[tzt^{-1}]_P$ is a unit of $R$, implying that $[tz^{-1}t^{-1}]_P\in R$. Bearing in mind the generating set of $R$ provided above, 
it follows that there are words $w_0,w_1,\dots,w_n\in\ol{A}^*$ for some $n\geq 2$ and $\varepsilon_i\in\{1,-1\}$, $1\leq i\leq n$, such that
$$
[tz^{-1}t^{-1}]_P = [w_0 t^{\varepsilon_1} w_1 t^{\varepsilon_2} w_2 \dots w_{m-1}t^{\varepsilon_m} w_m]_P
$$
holds, where no two consecutive $\varepsilon_i,\varepsilon_{i+1}$ are equal to $-1$, 
we have $[w_i]_G\neq 1$ whenever $\varepsilon_i\neq\varepsilon_{i+1}$, 
and we have $w_i\in W^*$ whenever $\varepsilon_i=1$ and $\varepsilon_{i+1}=-1$.
By the Normal Form Theorem for free products applied to $P$, this is possible only if 
$n=2$, $\varepsilon_1=1$, $\varepsilon_2=-1$, $[w_0]_G=[w_2]_G=1$ and  
$$[z^{-1}]_G = [w_1]_G,$$
which from the conditions above implies $w_1 \in W^*$. 
However, both elements of $G$ involved in the latter equality actually belong the subgroup $G_1$ of $G$ generated by $[z]_G$ and $H$, an isomorphic copy of $G_0$. 
Therefore we arrive at $[z^{-1}]_{G_0} = [w_1]_{G_0}$, but this is impossible by Britton's Lemma (see also \cite[Theorem IV.2.1]{LSch}) applied to the HNN extension $G_0$,
as the word $w_1$ may contain only occurrences of the letter $z$ but not of its inverse $z^{-1}$. This is a contradiction, whence $[tzt^{-1}]_M$ is not a unit of $M$.

Now consider the \Sch group of the $\H$-class of the element $[tzt^{-1}]_P$ in the (group-embeddable) monoid $R$, an isomorphic copy of the RU-monoid of $M$. 
By Proposition \ref{pro:general} (4), this group is isomorphic to
$$
U \cap [tz^{-1}t^{-1}]_P U [tzt^{-1}]_P.
$$
By using the Normal Form Theorem for free products again, we conclude that a typical word representing an element from $U$ is of the form 
$$
u_1 t v_1 t^{-1} u_2 t v_2 t^{-1} \dots u_m t v_m t^{-1} u_{m+1}
$$
where the words $u_1,\dots,u_{m+1}\in\ol{A}^*$, $v_1\dots,v_m\in\ol{B}^*$ are non-empty except possibly $u_1$ and $u_m$. Hence, the elements of $[tz^{-1}t^{-1}]_P U [tzt^{-1}]_P$ are
represented by words of the form 
$$
t z^{-1} t^{-1} u_1 t v_1 t^{-1} u_2 t v_2 t^{-1} \dots u_m t v_m t^{-1} u_{m+1} t z t^{-1},
$$
subject to the same conditions as above (note that if, for example, $u_1$ is empty then the prefix $t z^{-1} t^{-1} u_1 t v_1 t^{-1}$ of the previous word reduces to $t z^{-1} v_1 t^{-1}$).
So, upon employing the Normal Form Theorem for the third time, we conclude that an element of $U$ also belongs to the conjugate subgroup $[tz^{-1}t^{-1}]_P U [tzt^{-1}]_P$ of $P$ if and only if
$u_1,u_{m+1}$ are empty and there exist words $u_2,u_2',\dots,u_m,u_m'\in\ol{A}^*$ and $v_1,v_1',\dots,v_m,v_m'\in\ol{B}^*$ such that 
$$
[t v_1 t^{-1} u_2 t v_2 t^{-1} \dots u_m t v_m t^{-1}]_P = [t z^{-1} v_1' t^{-1} u_2' t v_2' t^{-1} \dots u_m' t v_m' z t^{-1}]_P.
$$
If $m\geq 2$, the latter condition is equivalent to the equalities $[u_i]_G=[u_i']_G$ for all $2\leq i\leq m$, $[v_i]_G=[v_i']_G$ for all $2\leq i\leq m-1$, $[v_1]_G = [z^{-1}v_1']_G$ and $[v_m]_G=[v_m'z]_G$
holding simultaneously. This implies $[z]_G = [v_1'v_1^{-1}]_G = [(v_m')^{-1}v_m]_G \in H$, a contradiction. Hence, $m=1$. In such a case, the previous condition reduces to 
$$
[t v_1 t^{-1}]_P = [t z^{-1} v_1' z t^{-1}]_P,
$$
which, on the other hand, holds for any words $v_1,v_1'\in\ol{B}^*$ such that $[v_1]_G=[z^{-1}v_1'z]_G$. From this we immediately conclude that 
$$
U \cap [tz^{-1}t^{-1}]_P U [tzt^{-1}]_P \cong H \cap [z^{-1}]_G H [z]_G \cong K_1.
$$
Thus we have that the considered \Sch group of $M$ is isomorphic to $K$, completing the proof.
\end{proof}


\small
\begin{ackn}
The research of the first named author is supported by the Personal Grant F-121 ``Problems of combinatorial semigroup and group theory'' 
of the Serbian Academy of Sciences and Arts. 
The research of the second named author was supported by the EPSRC Fellowship Grant 
EP/V032003/1 ``Algorithmic, topological and geometric aspects of infinite groups, monoids and inverse semigroups".
\end{ackn}
\normalsize


\end{document}